\documentclass[onefignum,onetabnum]{siamart190516}

\newcommand{\fix}[1]{{\color{black} #1}}
\usepackage{lipsum}
\usepackage{amsfonts}
\usepackage{graphicx}
\usepackage{epstopdf}
\usepackage{zref-base}

\usepackage{algorithmic}
\ifpdf
\DeclareGraphicsExtensions{.eps,.pdf,.png,.jpg}
\else
\DeclareGraphicsExtensions{.eps}
\fi

\newsiamremark{remark}{Remark}
\newsiamremark{hypothesis}{Hypothesis}
\crefname{hypothesis}{Hypothesis}{Hypotheses}
\newsiamthm{claim}{Claim}

\title{Collective motion planning for a group of robots using intermittent diffusion\thanks{
		\funding{This work was partially supported by grants NSF DMS-1830225, DMS-1620345, DMS-1720306, and ONR N00014-18-1-2852.}}}

\author{Christina~Frederick\thanks{Department of Mathematical Sciences, New Jersey Institute of Technology, Newark, New Jersey 10702
		(\email{christin@njit.edu}, \url{http://web.njit.edu/\string~christin/}).}
	\and Magnus~Egerstedt\thanks{Department
		of Electrical and Computer Engineering, Georgia Institute of Technology, Atlanta,
		GA, 30332 USA  (\email{magnus@gatech.edu}).}
	\and Haomin~Zhou\thanks{Department of Mathematics, Georgia Institute of Technology, Atlanta,
		GA, 30332 USA  (\email{ hmzhou@gatech.edu})}%
}

\usepackage{amsopn}

\input{cf_macros.sty}

\begin{document}
	\maketitle
	
	\begin{abstract}In this work we establish a simple yet effective strategy, based on intermittent diffusion, for enabling a group of robots
		to accomplish complex tasks,  shape formation and assembly.  
		We demonstrate the feasibility of this approach and rigorously prove collision avoidance and convergence properties of the proposed algorithms.
	\end{abstract}
	\begin{keywords}
		Path planning, multi-agent systems, optimal transport, intermittent diffusion
	\end{keywords}
	
	\begin{AMS}
		68Q25, 68R10, 68U05
	\end{AMS}

	\section{Introduction}
	Motion planning for multi-robot systems has drawn significant attention in recent years due to the emergence of a number of new application scenarios, e.g., 
	\cite{Latombe1990,Yan2013}.  Compared to single robot systems, multi-robot systems have many benefits, including spatial distribution, efficiency and robustness at completing a task due to division of labor, localization, information-sharing, redundancy, and potentially lower cost.
	On the other hand, motion planning for multi-robot systems must face significant challenges, such as collisions, deadlock due to the presence of local minima in the multi-objective functions from which the controllers are derived,
	and uncertainty introduced from the environment and stochastic effects in the system \cite{Koren}. Computationally, the motion planning problem can be NP-hard and not solvable in polynomial time even for some two\fix{-}dimensional cases \cite{Parker2002}. Furthermore, all of these difficulties are exacerbated when the robots are limited in capabilities, for example, short\fix{-}range communications. Addressing those existing challenges and satisfying the ever\fix{-}growing desire for new missions demand novel strategies and developments in both control engineering and their underlying mathematical theory.
	

	There is a vast literature for path planning that spans widely known methods, including graph\fix{-}based approaches such as A*, D*, or D*-lite, \cite{Dijkstra1959, Ferguson05, Hart1968, Koenig2004, Koenig, Koenig2005, Stentz1995, Likachev2008}, randomized algorithms such as Probabilistic Road Maps (PRM) \cite{Overmars92, Kavraki94,Amato96,Svestka1997}, and tree-search algorithms, including Rapidly-exploring Random Tree (RRT) \cite{Lavalle1998,LaValle2000, Fiorini1998, Park2013, Karaman}. These methods find trajectories, often optimal ones, by generating feasible paths defined by nodes on a lattice or random tree that characterizes the space of possible configurations. 
	
	Much progress has been made in adapting existing methods 
	to cooperative path planning problems for relatively small groups of robots \cite{Anderson2018, Cao1997, Ferrari2016, Jager2001, Lee2015, Ogren2001, Parker2002, Gerkey2004, Rekleitis2004, Luna2011, Agmon2012,  Yan2013, Hoy2015, Marcolino2009a} or the design of cooperative motion strategies without explicit preplanning of optimal paths \cite{Santos2018, Lu2014a,Egerstedt2016, Lee2015}. Readers are referred to a few survey papers on aerial swarm robots \cite{Chung2018} and collective behavior of multi-agent algorithms \cite{Rossi2018} that provide extensive lists of papers and summaries of many methods appeared in recent years. It is worth noting that one of the conclusions in \cite{Rossi2018} highlights the artificial potential functions (APF) method for its versatility, simplicity, scalability and high expressivity in swarm robots, and calls for new developments in both theory and algorithms that share the key properties of APF.
	
	APF is 
	proposed in \cite{Khatib1986}. It formulates the shape-formation problem as a problem of minimizing a potential composed of an attractive field, based on the desired shape, and a repelling field based on obstacles.  Designed originally for single-robot trajectories  \cite{Chanclou, Warren}, these theories and methods have been extended and improved upon over the past several decades, including the addition of simulated annealing and an extension to dynamic environments \cite{Rimon1992, Ge2002, Park2001a, Warrenb, Park2001, MinCheolLee2003}. Due to its simplicity and scalability, APF methods can handle large groups of robots, in which each robot regards others as obstacles, and higher dimensional problems efficiently. \fix{Recently, in \cite{Hsieh2008}, potential based methods were succesfully used to develop decentralized controllers for shape formation of a swarm of robots.} However, a well-known limitation of APF is the presence of local minima caused by the repelling forces of obstacles, leading to potential deadlocks.

	
	
	In this paper, we advocate designing motion planning methods for multi-robot systems by equipping APF with new ideas, such as intermittent diffusion, in recent developments in stochastic differential equations (SDEs) and global optimization, \fix{and Wasserstein gradient flows in probability space}.  We cast the motion planning for a group of robots as transporting one point-mass distribution (initial shape) to another point-mass distribution (target shape).  Unlike many existing motion planning problems in which each robot knows its target configuration, we do not assume that the robots know their precise destination, rather they must form the desired \fix{shape or} distribution in the end. We propose a strategy that produce{\color{blue}s} algorithms to control the group dynamics using carefully designed potentials and stochasticity.  Our contributions include
	\begin{enumerate}
		\item Design two dynamical systems, based on the idea of intermittent diffusion, that alternately produce the motion trajectories for a group of robots.
		\item Prove that our strategy produce\fix{s} collision{\color{blue}-}free motions in both continuous and discrete settings. 
		\item Prove the convergence to the desired shape by using optimal transport theory. Demonstrating the approach can overcome the problem of local minima and deadlocks.
	\end{enumerate} 
	
	It is worth mentioning that our approach is closely related to the theory of optimal transport \cite{Kantorovich2006, Villani2008}, a mathematical branch that finds many successful applications in 
	optics, econometrics, and computer graphics \cite{Ambrosio, Delon2013, Feng2016, Galichon2017,Merigot2011, Yu2007}, just to name a few.
	The connection between our approach and optimal transport theory {\color{blue} has} two different aspects. In theory, our proof for the convergence is through intermittent diffusion, whose proof relies on optimal transport theory. In {\color{blue} our} algorithm, the paths produced can be viewed as randomly perturbed particle motions, whose distribution density satisfies the well-known Fokker-Planck equation, which is regarded as a gradient flow of the relative entropy in optimal transport theory. This gradient flow viewpoint, which ensures its convergence to the Gibbs distribution, inspired our design. We use {\color{blue} the} target shape to create a Gibbs distribution, which guides the particle motions to produce the path.

	We also want to mention  that the proposed method differs from similar applications of optimal transport to robot path-planning \fix{\cite{Bandyopadhyay2014,DeBadyn2019, Krishnan2018}, in which either linear programming, quadratic programming}, or primal-dual method is used to identify the transport map. Instead of resorting to optimization methods in computations, we directly prescribe the gradient like dynamics for each robot to generate its trajectory using local information. The resulting equations can be simulated by robust numerical algorithms and executed \fix{efficiently}. 
	Furthermore, although our method shares a lot of similarities with APF, there are key differences. We add intermittent random perturbations in our dynamics to avoid deadlock, \fix{which overcomes the main limitation of APF at a moderately increased computation cost}. As in the method of evolving junctions (MEJ) \cite{Lu2014, Li2017}, it can be shown that the intermittent dynamics converges to the desired shape much quicker than the continuous white noise perturbations. \fix{In addition, the repelling fields from obstacles in APF methods affect the potential everywhere in the domain, while in the proposed method,
		each robot is viewed as a dynamically moving obstacle to the other robots, and its repelling effect is restricted to a small, local region}.  
	
	When viewing motion planning for multi-robot systems as transport of distributions, we note that there {\color{blue} is} recent work inspired by statistical physics \cite{Bertozzi2004, Zhang2018}, in which rigorous error estimates have been obtained between partial differential equations (PDEs) that model the swarm dynamics and the target distribution, enabling desired coverage performance. \fix{PDEs have also been used in \cite{Eren2017} to generate velocity fields that govern the motion-planning and incorporate collision avoidance. In \cite{Elamvazhuthi2019}, the controllability properties of the advection-diffusion equation are used to derive conditions on the target probability distribution that guarantee convergence in finite time for certain control inputs.} In addition, there are also other stochastic methods for path planning and control \cite{Kalra2006, Kaminka2010, Li2017b}.

	The paper is organized as follows. In \S \ref{sec:mathpre}, we present the basic optimal transport theory and Fokker-Planck equation that inspire us to design the dynamics. In \S \ref{sec:modelsetup}, we formulate the continuous problem in terms of a system of SDEs. The discretized problem is described in \S \ref{sec:implementation}. In \S \ref{sec:numerics} we provide numerical simulations of the shape formation problem for different shapes and different size groups. We provide theoretical guarantees for global convergence of the system and collision avoidance, both in the continuous and discrete settings in \S\ref{sec:theory}.
	
	\section{Relations between SDEs and optimal transport}\label{sec:mathpre}
	In this section, we briefly review the connections among stochastic differential equations (SDEs), Fokker-Planck equation, Gibbs distribution, free energy{\color{blue},} and optimal transport distance. These relations provide the theoretical foundation on which we design the dynamics for the motion planning of a group of robots.
	
	Let us consider a potential function $\Psi(y)$, \fix{in which $y \in \RR^{Nd}$ represents the  locations of $N$ robots in a bounded domain $\Omega\subset \RR^d$}. The white noise perturbed gradient flow refers to a SDE 
	\begin{equation} \label{sde1}
	d Y(t) = - \nabla \Psi(Y(t)) dt + \sigma d W(t),
	\end{equation}
	where $W(t)$ is the standard $\fix{Nd}$-dimensional Brownian motion and $\sigma$ a given constant. Denoting $\rho(t, y)$ the density function for the random variable $Y(t)$, the evolution of $\rho$ is governed by the Fokker-Planck equation according to the classical diffusion theory, i.e.
	\begin{equation}\label{fpe1}
	\frac{\partial \rho}{\partial t} = \nabla \cdot (\nabla \Psi(y) \rho) + \frac{1}{2} \sigma^2 \Delta \rho. 
	\end{equation}
	\fix{By directly plugging in the well known Gibbs distribution
		\begin{align} \label{gibbs}
			\rho^{*}(y) &= P^{-1}\exp(-2\Psi(y)/\sigma^2),\\ \text{where} \quad P&=\int_{\RR^{Nd}} \exp(-2\Psi(y)/\sigma^2)dx,
		\end{align}
		we see that $\rho^{*}$ is a steady state solution of \eqref{fpe1}, because it satisfies the equation and it is time independent. In other words, Gibbs distribution is an invariant measure of the system \eqref{sde1}. From the exponential form of $\rho^{*}(y)$, we also observe that the density $\rho^{*}(y)$ takes the largest value when $\Psi(y)$ reaches its global minimum. 
		
		We would like to remark that in order for the Gibbs distribution to be well-defined, $P$ in \eqref{gibbs} must be a finite number. This can be guaranteed by requiring that the potential function $\Psi(y)$ grows quadratically when $|y|$ tends to infinity. In this study, our interest is within a bounded region. Therefore, we can assume that $\Psi(y)$ is defined with such a property at infinity. }
	
	The understanding about the connections between Fokker-Planck equation and Gibbs distribution has been greatly enriched in the past few decades, thanks to the new developments in optimal transport theory. In short, defining the 2-Wasserstein distance between any two density functions $\rho^1(y)$ and $\rho^2(y)$ by
	\begin{equation*}
		\fix{(W_2(\rho^1, \rho^2))^2} = \inf_v { \int_0^1 \int_\Omega v^2 \rho dy dt},
	\end{equation*}
	where the velocity field $v(t,y)$ and density $\rho(t,y)$ satisfy the transport equation 
	\begin{equation}\label{tpe1}
	\frac{\partial \rho }{\partial t} + \nabla \cdot (v \rho) = 0,
	\end{equation}
	with boundary values given by 
	\begin{equation*}
		\rho(0,y) = \rho^1(y), \quad \rho(1,y) = \rho^2(y),
	\end{equation*}
	induces a metric in the probability density space and \fix{turns} the space into a Riemannian manifold to which one can apply various geometric operations. One of the most impactful results reveals that the Fokker-Planck equation is the gradient flow, with respect to the 2-Wasserstein metric, of a free energy given by 
	\begin{equation*}
		\mathcal{G}(\rho) = \int_\Omega \Psi (y) \rho (y) dy + \frac{1}{2} \sigma^2 \int_\Omega \rho (y) \log \rho (y) dy,
	\end{equation*}
	in which the first term is the potential energy while the second is called entropy \cite{Jordan1998}. Following the properties of gradient flow, one can prove that Gibbs distribution is the unique attractor of the Fokker-Planck equation \eqref{fpe1}, and its convergence rate to the Gibbs distribution is exponential, \fix{see Theorem 24.7 in \cite{Villani2008} for details}.  
	
	Our idea for motion planning is finding a potential such that the target shape is where the potential attains its global minima if shape formation is the task, or the target distribution is the Gibbs distribution if the goal is to move a group \fix{of} robots to a given distribution, The exponential convergence of \eqref{fpe1} to the Gibbs distribution from any initial distribution forms the basis that guarantees the success of planned motions. The potential is used in conjunction with \eqref{sde1} to create two deterministic dynamics that are used alternately to produce the trajectories for all robots. In the design, we must ensure that (a) the motions are collision{\color{blue}-}free; (b) there is no deadlock\fix{, and;} (c) the dynamics converge to the desir\fix{ed} shape. In the rest of this paper, we use shape formation as the task to illustrate our strategy. Its extension to the distribution case requires simple modifications, which will be omitted in the paper.
	
	{\bf Remark}: Besides the strategy that we propose here, there are different ways to apply optimal transport theory for motion planning. For example, one may view the robots as a collection of point mass and move them according to the transport equation \eqref{tpe1} for which the initial and target distributions are used as $\rho^1$ and $\rho^2$ respectively. This amounts to find{\color{blue}ing} a velocity $v$ while maintaining point masses throughout the optimization procedure. We do not adopt this view in this paper. Instead, we directly design dynamics based on formulation \eqref{sde1}, because the resulting algorithm is simple and efficient in implementation, yet has desirable properties that can be rigorously proved. 
	
	It is worth noting that the convergence to the Gibbs distribution for the solution of the Fokker-Planck equation does not have a direct guarantee for the convergence of the SDE \eqref{sde1} to a desirable shape. The subtlety lies in the fact that the convergence for the SDE is only in the distribution sense. The solution of \eqref{sde1} with a positive constant $\sigma$ never settles down asymptotically. To make the solution converge, one has to reduce the value of $\sigma$ gradually to zero, which is precisely the idea used in simulated annealing. However, it is well known that the reduction rate of $\sigma$ must be slower than a logarithmic function in time to avoid local traps. To speed up the convergence, we borrow ideas from intermittent diffusion \cite{Chow2013}, a stochastic strategy developed for global optimization that can improve the convergence with the probability of success increased to 1 as a geometric sequence, a rate that is much faster than logarithmic functions.  Besides, directly applying random perturbations to the motions can cause wasteful jittering effects which we want to avoid in our design.
	
	\section{Model setup}\label{sec:modelsetup}
	
	Suppose $\Gamma\subset \RR^{2}$ is a set of spatial locations that form a desired shape, and consider the trajectories of $N$ robots given by
	\[X({t}) = (X^{1}({t}), \hdots, X^{N}(t)),\] 
	where $X^{{i}}(t)$ is a curve in $\RR^{2}$ describing the position of the $i^{\text{th}}$ robot at time $t \ge 0$.  
	The objective is to produce paths $\{X(t)\}_{0\leq t\leq T}$ from an initial state $X(0)=X_{0}$ to a final state $X(T)$ such that $X(T) \in \Gamma$. Our strategy is to design modified gradient flows whose solutions prescribe the path $X(t)$ for all robots. 
	
	In order to do so, we first introduce a shape function \fix{$F(X)$ that is smooth and has a global minimum only for $X \in \Gamma$}. A convenient choice, among many candidates, is the distance function,
	\begin{align}
		F(X) =  \frac{1}{N}\sum_{i=1}^{N} \mu(X^{i}); \qquad \mu(X^i)= \min_{X'\in \Gamma} \|X^i-X'\|^{2}, \labeleq{distance}
	\end{align}
	where $\|\cdot\|$ is the Euclidean norm: $\|x\| = \sqrt{x_1^2+x_2^2}$. \fix{  Then}, $F(X)$ is a non-negative function achieving its minimum only when $\cup_{i=1}^N X^i \subset \Gamma$. Figure \ref{fig:density} illustrates the level-sets of $F(X)$ corresponding to two different target shapes. 
	
	We also introduce a penalty function $G(X)$ 
	that takes a large value when $X$ exhibits undesirable behavior. In multi-robot systems, one of the main objectives is to ensure that the trajectories are {\it collision{\color{blue}-}free}, meaning the pairwise distances $ \|X^{{i}}(t)-X^{j}(t)\|$, $j\neq i$ must be larger than a given positive value $r$, for all $t>0$. For example, we can select the \fix{penalty $G(X)$} as the following smooth, ``repelling'' function that peaks when the pairwise distances are small,
	\begin{align}
		G(X) =\begin{cases} G_0 \displaystyle \sum_{i=1}^{N} \sum_{j \neq i}  \varphi(\|X^{i}-X^{j}\|/2),& \text{if } \|X^{i}-X^{j}\| < R,\\
			0,& \text{ otherwise}
		\end{cases},
		\labeleq{repelling}
	\end{align}
	\fix{where the function $\varphi \in C^{1} (0,\infty)$ can be chosen as a {\it decreasing} function having the following properties
		\begin{align}
			\lim_{x\rightarrow R^{-}}\varphi(x)=\lim_{x\rightarrow R^{-}}\varphi'(x) = 0. \labeleq{gpropR}
	\end{align}}

	\fix{For instance, $\varphi(x)= \frac{1}{x} \exp{(\frac{-1}{R^2 - x^2}})$ can be picked to satisfy the requirements.}
	Here, the constant $R>r$ is related to the sensing radius of each robot, and \fix{the constant $G_0$ is calibrated according to the initial positions of robots and to achieve desirable dynamics.} 
	Further constraints on the system, such as obstacle avoidance, can also be easily included. To simplify the presentation, we do not consider obstacle avoidance in this paper.  
	
	Combining the shape function \refeq{distance} with the penalty function \refeq{repelling}, we obtain the potential function
	\begin{equation}
	\Psi(X) = F(X)+G(X). \label{energy}
	\end{equation}
	Then the trajectories of the robots are primarily generated by the gradient flow that minimizes $\Psi(X)$, i.e. 
	\begin{equation}
	\frac{dX^i(t)}{dt} = - (\nabla \Psi(X(t)))_i. \label{gradient}
	\end{equation}
	Following it, the robots get to the desired shape when $F(X)=0$, while minimizing $G(X)$ helps to spread out their locations in addition to avoid collisions.

	\begin{figure}
		\begin{center}
			\includegraphics[width=.5\linewidth]{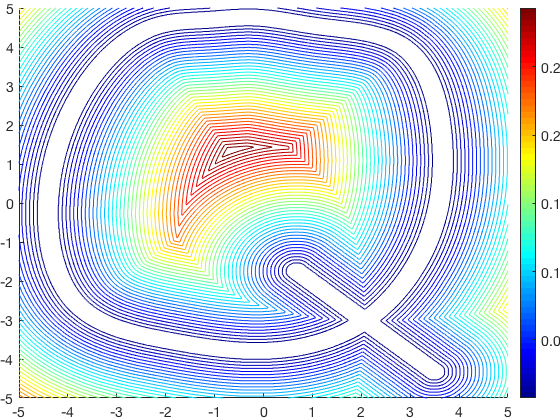}\hfill \includegraphics[width=.5\linewidth]{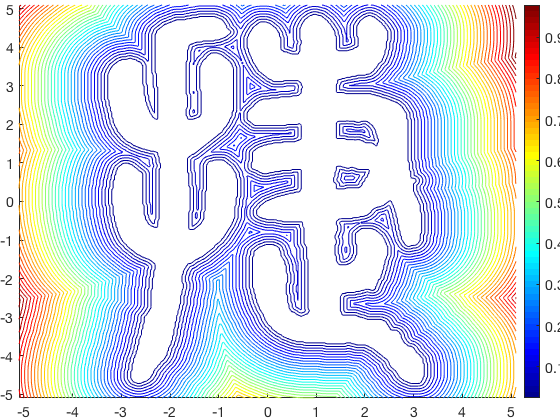}
			\caption{\fix{Surface plot and contour plot of the distance function $\mu(X^i)$  in \refeq{distance} for two different shapes $\Gamma_{j}$, $j=1,2$.}}\label{fig:density}
		\end{center}
	\end{figure}
	
	However, the path generated by such a simple gradient flow may suffer a well known shortcoming that the \fix{trajectories can get trapped in locations corresponding to} local minimizers. To overcome this limitation, we use ideas from intermittent diffusion. More precisely, we intermittently add random perturbations to \eqref{gradient}, leading to the following SDEs,
	\begin{equation}
	\begin{cases}
	dY^i(t)=-(\nabla \Psi (Y(t)))_idt+\sigma(t) dW({t}), & t>0, \\
	Y(0) = Y_{0},
	\end{cases}
	\label{SDE}
	\end{equation}
	where $W(t)$ is the standard Brownian motion in $\RR^{2}$ and $\sigma(t)$ is a piecewise constant function alternating between zero and a positive value, i.e.
	\begin{equation}
	\sigma(t) = \begin{cases}
	0 & \text{if $t \in [S_k, T_k]$} \\
	\sigma_k & \text{if $t \in [T_{k-1}, S_k]$}.
	\end{cases} \label{sigma}
	\end{equation}
	Here we partition $[0,T]$ as $\cup_{k=1}^K([T_{k-1},T_k])$ with $T_0 = 0$, $T_N = T$ and $S_k \in [T_{k-1}, T_k]$.
	
	We want to highlight that the random perturbations are added to the gradient flow
	to avoid trajectories being trapped at local minimizers. Therefore, the constant $\sigma_k$
	doesn't have to be small. This is different from the choice used in simulated annealing, in which the corresponding coefficient, also called temperature, must go to zero asymptotically. The effectiveness of random perturbations can be verified by numerical experiments and comes with guarantees based on optimal transport theory. More precisely, the solution of \eqref{SDE} converges to the global minimizer in the distribution sense according to the Gibbs distribution.
	The Gibbs distribution is an invariant measure of the system \eqref{SDE}, and $\rho^{*}(X)$ takes the largest value when $\Psi(X)$ reaches its global minimum. Further details of the theory are provided in \S \ref{sec:theory}.
	
	Unlike many other applications of SDEs, it is important to emphasize that the random portion of the solution $Y(t)$ when $t \in [T_{k-1},S_k]$ is not used as the trajectories for the robots due to inefficient jittering motions. Instead, $Y(t)$ is only computed virtually to create the vector $Y(S_k)$, denoted as $\hat{Y}$ in the rest of the paper, of intermediate positions to move the robots to. Once this position $\hat{Y}$ is computed, we define another objective function 
	\begin{equation}
	\hat{F}(X) = \frac{1}{N}\sum_{i=1}^{N}   \|X^{i}-\hat{Y}^i\|^{2}.
	\end{equation}
	Using it together with $G(X)$, we create another gradient flow 
	\begin{equation}
	\frac{dX^i(t)}{dt} = - \left(\nabla \left(\hat{F}(X(t))+G(X(t))\right)\right)_i. \label{gradient2}
	\end{equation}
	In the end, the path $X(t)$ of the robots is generated by alternating between two gradient flows \eqref{gradient} and \eqref{gradient2}. The implementation of the method is given in the next section.

	\section{Implementation}\label{sec:implementation}
	
	The gradient flows and the SDEs presented in the previous section must be solved numerically when calculating the path. We employ the simple Euler scheme to do so in this paper. More precisely, we compute  
	\begin{equation}
	X^i_{n+1} = X^i_n - \Delta t (\nabla \Psi(X_n)))_i, \label{euler1}
	\end{equation}
	where $\Delta t$ is the step size, $\Psi(X)$ takes $F(X) + G(X)$ for \eqref{gradient} and $\hat{F}(X) + G(X)$ for \eqref{gradient2} respectively. The SDEs \eqref{SDE} is discretized as 
	\begin{equation}
	Y^i_{n+1} = Y^i_n - \Delta t (\nabla \Psi(Y_n))_i + \xi_n \sqrt{\Delta t}, \label{euler2}
	\end{equation}
	where $\xi_n \in \RR^{2}$ is a normally distributed random vector generated at each iteration. 
	
	As mentioned in the previous section, the path is generated by alternating between \eqref{gradient} and \eqref{gradient2}. This is implemented by repeating a 2-step strategy. In the first step, the robots are 
	moved, using \eqref{gradient2}, toward temporary destinations computed by a simulation of \eqref{euler2}. After the temporary locations are reached, the second step has the robots follow
	\eqref{gradient}
	toward the desired shape. The robots then repeat the two steps until the task is accomplished. Details are presented in Algorithm \ref{algorithm}, and the computed descent directions in two different iterations are plotted in Figure \ref{fig:quiver}. Again, we want to re-iterate that
	$Y_n$ is not part of the trajectories. They are 
	computed only virtually to generated the intermediate positions $\hat{Y}$. 


	\begin{algorithm}
		\floatname{algorithm}{Algorithm}\caption{Intermittent diffusion based motion-planning}
		\renewcommand{\thealgorithm}{}
		\label{algorithm}
		\begin{algorithmic}[1]
			
			\STATE{\bf Initialization}: Given a feasible initial configuration $X_0$ in a computational domain $\Omega = [-M, M]^d$, and a tolerance $\epsilon>0$. Pick ID parameters $(\alpha, \beta)$, a small number $\tau >0$ and a positive integer $s_{max}$ as the maximum iteration number in step 3.  Set $X_{opt}=X_{0}$, $n=0$. 
			\STATE {\bf Virtual diffusion:} If $\Psi(X_{opt})>\epsilon$, set $k=k+1, m=0$. Generate two random positive numbers $d, t \in (0,1)$ and set $\sigma=\alpha d$ and $V =\beta t $. Define $Y^{0}=X_{opt}$, and perform the following simulation for $m\Delta t\leq {V}$:
			\begin{align}{\color{blue}
					Y_{m+1}^{i}=Y_{m}^{i}- (\frac{1}{N}\nabla \mu (Y_{m}^{i}) + (\nabla  G (Y_{m}))_{i} )\Delta t + \sigma \xi^i_m \sqrt{\Delta t},}\end{align}
			in which $\xi^i_m$ is a standard normal random variable. Record the final locations $\hat{Y}=Y_{m}$.
			\STATE {\bf Gradient descent toward $\hat{Y}$:} For $1 \leq i \leq N$, define 
			\begin{align*}
				\hat{F}_{i}(X^{i}) &= \frac{1}{N}\|X^{i} - \hat{Y}^{i}\|^{2}. 
			\end{align*}
			Set $s_0=n$. Compute the following iterations until $\max_{i}\|X^{i}_{n+1}-X^{i}_{n}\|<\tau $ or $n>s_0+s_{max}$, \begin{align*}
				X_{n+1}^{i}&=X_{n}^{i}-   ( \nabla \hat{F}_{i} (X^{i}_{n})+(\nabla G (X_{n}))_{i} )  \Delta t. \end{align*}
			If { $X^{i}_{n+1} \not\in\Omega$}, set $X^{i}_{n+1} = X^{i}_{n+1}-2\text{sgn}(X^{i}_{n+1})\mod(\|X^{i}_{n+1}\|_\infty, M)$.

			\STATE {\bf Gradient descent toward $\Gamma$:} Calculate the following iterations until $\|X_{n+1}^{i}-X_{n}^{i}\|<\tau $:
			\fix{
				\begin{align*}
					X_{n+1}^{i}&=X_{n}^{i}-  ( \frac{1}{N}\nabla \mu (X^i_{n}) + (\nabla  G (X_{n}))_{i} )  \Delta t \end{align*}}
			If { $X^{i}_{n+1} \not\in\Omega$}, set $X^{i}_{n+1} = X^{i}_{n+1}-2\text{sgn}(X^{i}_{n+1})\mod(\|X^{i}_{n+1}\|_\infty, M)$.
			If $\Psi({X}_{n})<\Psi(X_{opt})$, set $X_{opt}={X}_{n}$.

			\STATE Repeat steps 2,3, and 4 until $\Psi(X)<\epsilon$. 
			
		\end{algorithmic}
	\end{algorithm}

	This algorithm is a practical modification of the theory developed in the later sections. The main difference is in the diffusion stage of the algorithm, the aim is to produce trajectories that are influenced by both the desired shape and random noise. This procedure is performed offline to save resources; a random path simulated by a robot may be costly even if the ending location is close to the starting position of the robot. Instead of a random path, it is more efficient for the robots to move directly toward these temporary locations. Therefore, in the implementation, each robot moves to its computed destination following a gradient flow, without regard for the shape density function. By doing this, the energy of the system will possibly be increased. This is reflected in the variance of the energy functional in Figure \ref{fig:energy}. 
	
	In \S\ref{sec:theory}, we shall prove that the Algorithm \ref{algorithm} generates \fix{a} guaranteed collision\fix{-}free path for each robot that converges to the desired shape. Before doing so, we present a few numerical experiments to illustrate the performance in the next section. 
	
	\begin{figure}
		\begin{center}
			\includegraphics[width=.5\linewidth]{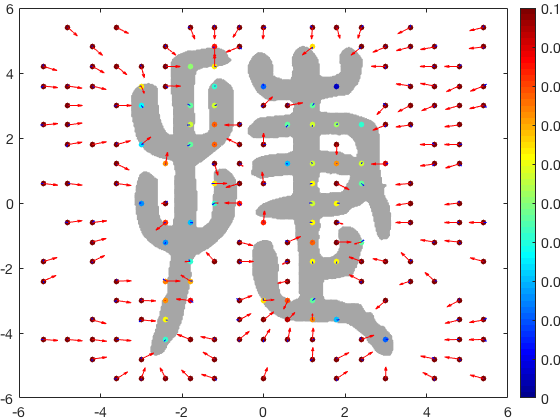}\hfill \includegraphics[width=.5\linewidth]{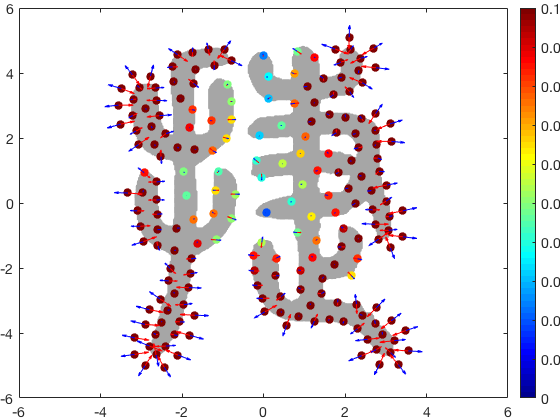}
			\caption{\fix{The arrows indicate the directions $(\nabla G(X))_{i}$ (blue) and $(\nabla F(X))_i$ (red) at the beginning of the iterations (left) and midway through the iterations (right).}} \label{fig:quiver}
		\end{center}
	\end{figure}

	\section{Numerical Results}\label{sec:numerics}

	\begin{center}

		\begin{table}
			\centering
			\caption{Simulation parameters}
			\pgfplotstabletypeset[normal, string type,
			columns/Value/.style= {column name={ Value}, column type=l},
			]{ %
				Symbol & Description  &Value\\ 
				{$G_{0}$}	& Repelling function amplitude  &{.01}  \\ 
				{$R$} & Robot sensor radius  & {$10r$} \\ 
				{$\Delta t$}	& Time step &  {$.1 r$}  \\ 		
				{$\alpha $}	& ID Diffusion scale  &  {$r$}  \\ 
				{$\beta$}  & ID Time scale & {10} \\
				{$M$}  & Computational domain size &  6 \\ 
			}
		\end{table}
	\end{center}
	
	In our numerical experiments, we confine the robots in a square domain given by $\Omega = [-M, M] \times [-M, M]$. We \fix{assume} that each robot has knowledge of its location $X^i$, the gradient of the shape function  $(\nabla F (X))_{i}=\nabla F_{i} (X)$\fix{,} and a sensing radius $R$, meaning that a robot can only detect other robots if they are within a circular region centered at $X^i$ with radius $R$. This $R$ is also the parameter we use in $G(X)$: 
	\[ G(X) = G_{0}\sum_{i=1}^{N} \sum_{\substack{j=1\\j\neq i}}^{N} \cot(\pi/2(\|X^{i} - X^{j}\|^{2})/R^{2}).\]
	
	We note that this choice of $G$ is different from the function we mentioned in Section \ref{sec:modelsetup}, demonstrating the flexibility of choosing $G$.

	We evaluate the success of the algorithm by determining if the robots are in the desired region, distributed uniformly, and if the nearest\fix{-}neighbors difference is minimized.

	The numerical tests are performed on two shapes. The first shape consisting of points in the set $\Gamma_{1}$, corresponds to a handwritten letter `Q'. In this case, the closed loop feature poses difficulties. The second shape consisting of points in the set $\Gamma_{2}$, is a Chinese character, \fix{pronounced} as `JIE', with multiple complicated strokes and two disconnected components. The initial positions for the robots are either clustered at a corner (demonstrated for shape $\Gamma_1$) or randomly distributed in the domain (demonstrated for shape $\Gamma_2$). 
	The time evolution, shown in two cases in Figure \ref{fig:Q-corner-Jie-rand}, indicates that the robot trajectories driven by our proposed algorithm drive the robots to the desired shapes without suffering from congestion or getting stuck at local minimizers.  

	To test the scalability of our algorithm, we varied the size of the robot radius (resulting in different values of $N$). The choice of $N$ is based on a-priori knowledge that there is a global minimum with $N$ robots positioned entirely in the desired shape, determined by trial and error.

	\begin{table}\caption{Final objective function value for both \fix{target} shapes and varied \fix{robot radii} ($r$), starting from a random initialization.}
		\centering
		\pgfplotstabletypeset[normal,
		columns/r/.style={string type, column type=r, column name= $r$},
		columns/nRobots/.style={precision=0, column type=r, column name= $N$}, columns/energyID/.style={column name=$\Psi(X_{ID})$},  
		columns/energyGD/.style={column name=$\Psi(X_{GD})$}
		]{
			r &nRobots  &energyID &energyGD  \\
			\topmidheader{4}{`Q' }
			.1 &50&	    0.02537064&	 0.02808049\\
			.05 &150&	0.02905039&		  0.02930737\\
			.01 &1000&		0.00065954&  0.00222986\\
			\midheader{4}{`Jie' }
			.1 &200	  	  &0.12890581			&.13272578\\
			.05 &400	&	0.05963590		&0.06544953\\
			.01 &3000	&0.00701573 		&0.01647748\\
		}
	\end{table}

	From the experiments, we observe that the faster convergence occurs with a random initial configuration that minimizes congestion from the start and provides the robot group immediate access to all sides of the target shape. When robots are initialized in a cluster near one end of the domain, they risk stagnating near the corner of the shape and missing entire sections of the shape unless intermittent diffusion becomes active.

	\begin{figure}
		\caption{Time evolution generated by the Intermittent Diffusion based Algorithm \ref{algorithm}. The color of the $i$th robot indicates the value of $\fix{\mu(X^{i})}$.}\label{fig:Q-corner-Jie-rand}
		\begin{center}
			\includegraphics[width=.5\linewidth]{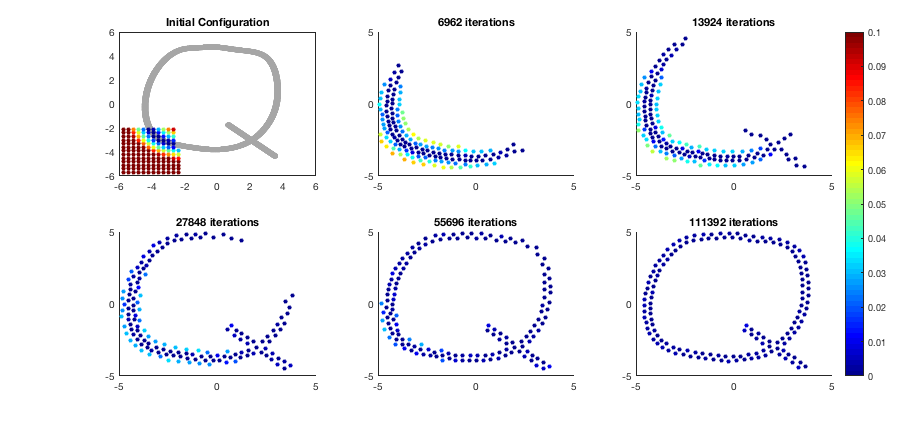}\includegraphics[width=.5\linewidth]{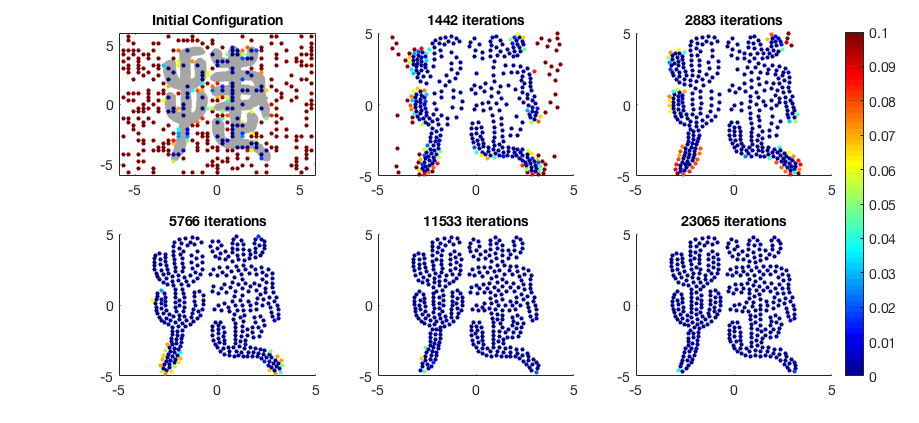}
		\end{center}
	\end{figure}

	We compared our method to a standard gradient descent with \fix{the} potential $\Psi$, which is the result of APF. From the energy plots shown in Figure \ref{fig:energy}, it is clear that gradient descent (APF) alone leaves \fix{some} robots trapped in local minima. After about 2000 iterations, the congestion caused by the gradient descent iterations is not resolved. Furthermore, the energy decays at a much slower rate than in the iterations produced by \fix{Algorithm \ref{algorithm}}.

	\begin{figure}
		\caption{Plots of the potential function $\psi(X)$ versus the iteration number  for both the gradient descent algorithm and \fix{Algorithm \ref{algorithm}}. \fix{Robots are initialized in the corner of the domain (for $`Q')$ and randomly (for `Jie') and $r=.05$}. }\label{fig:energy}
		\begin{center}
			\includegraphics[width=.5\linewidth]{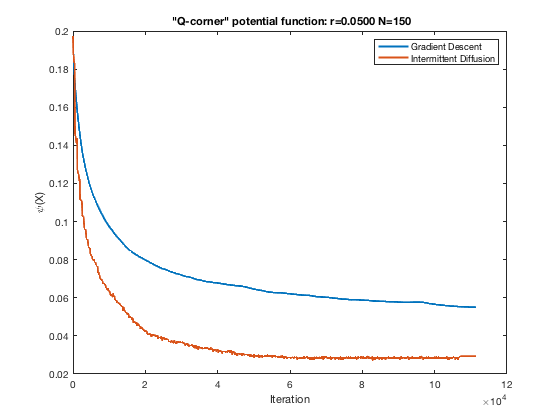}\hfill\includegraphics[width=.5\linewidth]{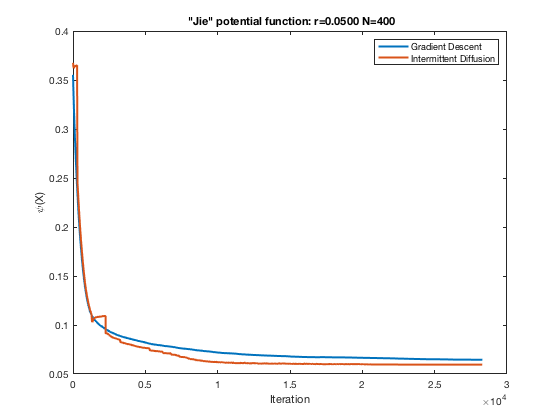}
		\end{center}
	\end{figure}

	\section{Mathematical Underpinnings}\label{sec:theory}
	
	In this section, we justify theoretically that the generated path using the proposed method can achieve the desired shape while maintaining collision{\color{blue}-}free motions. We start with the collision{\color{blue}-}free property first. 
	
	Our model determines the trajectories of the robots based on
	two different gradient flows, \eqref{gradient} and \eqref{gradient2} respectively. In both cases, the energy functional $\Psi(X)$ consists of a potential $F(X)$ (or $\hat{F}(X)$) that attracts the robots to the destinations, and the repelling function $G(X)$ that keeps them away from each other. In our theoretical study, it suffices to consider a general potential $F$ that is differentiable, \fix{is} bounded, and has minimizers only at the desired regions. In this general setting, the governing equation for the path is still given by the gradient flow presented in \eqref{gradient}. 
	
	

	\subsection{Continuous time collision avoidance}\label{sec:continuouscollisionavoidance}
	Recall that the location of the $i^{\text{th}}$ robot is given by $X^{{i}}$, and the set of admissible robot coordinates is $\mathcal{X}$, where
	\begin{align*}
		\fix{\mathcal{X}=\{(X^{1}, \hdots, X^{N}) \mid X^i\in \Omega, \inf_{ i,j\neq i} \|X^{{i}} - X^{j}\|> r >0\}.}
	\end{align*}
	We note that the repelling function $G(X)$ satisfies \refeq{repelling}, for a function $\varphi$ satisfying  \refeq{gpropR}, \fix{which implies $G(X)$ is a $C^1$ function on $\mathcal{X}$.}
	
	
	\fix{Let $X_0 \in \mathcal{X}$ be the initial robot locations with smallest pairwise distance given by
		$$
		m_0 := \min_{i \neq j} \{\|X_0^i - X_0^j\|\} >r.
		$$
		Suppose that the decreasing function $\varphi(x)$ satisfies
		\begin{equation} \label{phicond}
		G_0 \varphi(r) > E_0 := \Psi(X_0).
		\end{equation}
		This can be achieved for $\varphi(x)$ satisfying the property $$\varphi(r)> N^2 \varphi(m_0) + \frac{2M}{G_0}.$$
	}
	
	Then we have the following 
	theorem.
	

	\fix{
		\begin{theorem}  \label{thm1}
			For any trajectory $X(t)=(X^{1}(t), \hdots, X^{N}(t))$, $N>1$ generated by \eqref{gradient} with initial position $X(0)=X_{0}\in \mathcal{X}$, the inequality \begin{equation}\label{eq:minDistance}
			\inf_{i,j\neq i} \|X^{{i}}(t) - X^{j}(t)\|^{2}  >  r^{2}
			\end{equation}
			holds for all $t>0$.
		\end{theorem}
		\begin{proof}
			The function $\Psi (X(t))$ is non-increasing along the solution of \eqref{gradient} since it satisfies
			\begin{align*}
				\frac{d}{dt}\left(\Psi(X)\right)& = \sum_{i=1}^N(\nabla \Psi(X))_i \cdot \frac{dX^i}{dt}\\
				&= -\sum_{i=1}^N\|(\nabla \Psi(X))_i\|^{2} \leq 0.
			\end{align*}
			
			
			Assume there is a time $t^{*}>0$ such that $\|X^{{i}}(t^{*})-X^{j}(t^{*})\|^{2} \leq {r^{2}}$ for some $i, j\neq i$, then 
			\begin{align*}
				\Psi(X(t^{*})) &= F(X(t^{*}))+G(X(t^{*}))
				\geq G_0 \varphi(r) > {E_{0}} = \Psi({X_0}).
			\end{align*}
			Here we used that $F$ is non-negative by construction. This is a contradiction, because $\Psi(X(t))$ is non-increasing, so we must have $\Psi(X({t}^{*})) \le E_{0}$.

		\end{proof}
	}
	
	\subsection{Discrete time collision avoidance}\label{sec:discretecollisionavoidance}

	Equation \eqref{gradient} and \eqref{gradient2} are solved in discrete time using the iterations
	\begin{align}
		X^i_{n+1} &= X^i_{n} - \left(\nabla\Psi(X_{n})\right)_i \Delta t,  \labeleq{gditerates}
	\end{align}
	where $X_{n} \simeq X({t_{n}})$ and $t_{n}=n\Delta t$ for some \fix{fixed time step $\Delta t>0$}. It is known that the Euler scheme converges to the continuous solution if $\nabla \Psi$ is $L-$Lipschitz continuous in space. This ensures no collision in the discrete case when the step size is small enough. In the next theorem, we present such a result, and prove it by using a standard argument from \cite{Nesterov2004}.

	\begin{theorem} \label{thm2}
		Suppose $\Psi\in C^{1}(\mathcal{X})$ is a positive function that is bounded below and $\nabla\Psi$ is $L-$ Lipschitz continuous in space. Then, if $\Delta t \leq \frac{1}{L}$, one step of the gradient method \refeq{gditerates} will not increase the objective function $\Psi$, that is $\Psi(X_{n+1})\leq \Psi(X_{n})$.
	\end{theorem}
	
	\begin{proof} Denote the Euclidean inner product by $\langle X, Z\rangle = \left(\sum_{i=1}^N X^i\cdot Z^i\right)^{1/2}$ For $X, Z\in \mathcal{X}$, we can express $\Psi (Z)- \Psi(X)$ by
		
		\begin{align*}
			\Psi (Z)- \Psi(X) &=\int_{0}^{1} \langle \nabla \Psi(X + \tau(Z-X)), Z-X\rangle d\tau \\
			&= \langle \nabla \Psi(X), Z-X\rangle
			+\int_{0}^{1} \langle \nabla \Psi(X + \tau(Z-X)) - \nabla \Psi(X), Z-X\rangle d\tau . 
		\end{align*}
		This results in 
		\begin{align*}
			\Psi(Z) - \Psi(X) -\langle \nabla \Psi(X), Z-X\rangle
			&= \int_{0}^{1} \langle \nabla \Psi(X + \tau(Z-X)) - \nabla \Psi(X), Z-X\rangle d\tau \\
			&	\leq \int_{0}^{1}  \sum_{i=1}^N\| (\nabla \Psi(X + \tau(Z-X))_i) - (\nabla \Psi(X))_i\| \sum_{i=1}^N \| Z^i-X^i \| d\tau \\
			\leq   \int_{0}^{1}& L \tau \sum_{i=1}^N \| Z^i-X^i\|^2 d\tau =\frac{L}{2}\sum_{i=1}^N \| Z^i-X^i\|^{2}.
		\end{align*}
		Taking $Z^i = X^i_{n+1} = X_n^i - (\nabla \Psi (X_n))_i \Delta t$, we have
		\begin{align*}
			\Psi (X_{n+1}) &\leq \Psi(X_n) -\Delta t \sum_{i=1}^N\| (\nabla \Psi(X_n))_i\|^{2} + \frac{L\Delta t^{2}}{2}\sum_{i=1}^N \|(\nabla \Psi (X_n))_i\|^{2}\\
			& =  \Psi(X_n) -\Delta t (1- \frac{L\Delta t}{2})\sum_{i=1}^N \|(\nabla \Psi (X_n))_i\|^{2}.
		\end{align*}
		Therefore $\Psi (X_{n+1})\leq \Psi(X_n)$ if $\Delta t\leq \frac{2}{L}$.
	\end{proof}
	
	\fix{We remark that $\nabla \Psi (X) = F(X) + G(X)$, with $G(X)$ being defined through $\varphi(x)$, satisfies the $L$-Lipschitz condition in the domain of interest, because $\varphi(x)$ is a $C^1$ function on the closed interval $[r,2M]$. The Lipschitz constant $L$ depends on the choice of $\varphi$, the size of computational domain $\Omega$, and the number of robots in the group.
		
		\begin{corollary}
			The discrete trajectory $X_n$ computed by \eqref{euler1} satisfies 
			\begin{equation}\label{eq:minDistance2}
			\inf_{i \neq j} \|X_n^{{i}} - X_n^{j}\|^{2}  >  r^{2},
			\end{equation}
			for all $n \geq 0$, provided $E_0 = \Psi(X_0) < G_0\varphi(m_0)$. 
		\end{corollary}
		
		The proof of this corollary follows directly from the proof of Theorem \ref{thm1} and the result of Theorem \ref{thm2}. }
	
	

	\subsection{\fix{Convergence to the global minima in probability}}
	As described in the model, the goal of introducing \eqref{gradient2} is to move the robots to the intermediate locations generated by the SDEs \eqref{SDE}. Therefore, the convergence of the trajectories to the desired shape means that the solutions of \eqref{SDE} march to the global minima of $\Psi(x)$, which is guaranteed by the theory of optimal transport. More precisely, the idea of combining \eqref{gradient} and \eqref{gradient2} comes from the intermittent diffusion. Together, the dynamics can be equivalently described by a uniform formula given in \eqref{SDE}, in which \eqref{gradient} is performed when $\sigma = 0$, and \eqref{gradient2} reaches the same spatial locations as \eqref{SDE} when $\sigma$ is not zero. \fix{Hence the question of whether or not $X(t)$ converges to the desired shape can be investigated by examining the distributions of trajectories in \eqref{SDE}.} 
	
	We recall from Section \ref{sec:mathpre} that \fix{the probability density function $\rho(y,t)$ of the stochastic process $Y(t)$ from \eqref{SDE} evolves according to the Fokker-Planck equation},
	which is a transport equation when $\sigma =0$, and a diffusion equation when $\sigma > 0$. \fix{In the diffusion case, the asymptotic solution, also called the steady state, is the Gibbs distribution defined in \eqref{gibbs},
		suggesting the probability that $X(t)$ is within the attractive neighborhood, denoted by $\hat{U}$, of the global minimum of $\Psi$ is positive if $t$ is large enough. Here $\hat{U}$ is defined as the neighborhood of $\Gamma$, in which the trajectory of the gradient flow \eqref{gradient} with any initial configuration $X_0\in \hat{U}$ satisfies $\lim_{t\rightarrow \infty}\mu( X^i(t))= 0$, where $\mu$ is the distance function to $\Gamma$ defined in \refeq{distance}. 
		By the subsequent gradient flow \eqref{gradient}, $X(t)$ remains inside of the target $\Gamma$ or moves arbitrarily close to it. This suggests that there is a positive probability that $X(t)$ is within a small neighborhood of $\Gamma$ after one cycle of intermittent diffusion ($\sigma$ taking positive and then zero values once) is also positive. Repeating the cycle of intermittent diffusion, we obtain the following convergence theorem.}

	\begin{theorem} \label{thm4}
		Suppose $\Psi(x)$ attains its global minima on a set $\Gamma$ of positive Lebesgue measure, and \fix{let $U \subseteq \hat{U}$ be a small neighborhood of $\Gamma$}. Then for any $0<\eta<1$ there exist constants $T^{*}>0$, $\sigma_{0}>0$ and $K_{0}>0$ such that if $(T_{i}-S_{i})>T^{*}$, $\sigma_{i}<\sigma_{0}$  for $1\leq i \leq K$ and $K>K_{0}$, the solution $X_{opt}$ calculated by Algorithm \ref{algorithm} satisfies
		\begin{align*}
			\mathbb{P}(X_{opt}\in U)\geq 1-\eta.
		\end{align*}
		where $\mathbb{P}$ is the probability \fix{function}.
	\end{theorem}
	
	The proof of this theorem essentially follows the same steps as the proof \fix{in \cite{Chow2013}, in which the convergence of the density is considered in the $L^1$ sense by using the Csiszar-Kullback inequality (see Remark 22.12 in \cite{Villani2008}). For the completeness of this paper, we present a sketch of the proof modified to guarantee convergence in the $W_2$ sense. 
		
		\begin{proof} By the construction of $\Phi(X)$, its value is non-negative and reaches the minimum $0$ only if $X \in \Gamma$. This suggests that the Gibbs distribution $\rho^*(y)$ attains its maximum when $y \in \Gamma$. 
			Since $\Gamma$ has a positive Lebesgue measure, so does $\hat{U}\supset \Gamma$. Hence there exits a positive constant $\nu$ such that
			$$
			\int_{y \in \hat{U}} \rho^*(y) dy = 2\nu > 0.
			$$
			In fact, $2\nu \in (0,1]$ approaches $1$ when $\sigma$ tends to $0$ according to the property of Gibbs distribution. 
			
			From Theorem 24.7 and discussions following in Example 24.8 and Remark 24.12 in \cite{Villani2008}, we have 
			$$
			W_2(\rho(y,t), \rho^*(y)) < C e^{-\lambda t},
			$$
			where $C$ and $\lambda$ are constants, $\lambda$ is related to the well-known Log-Sobolev inequality (see Definition 21.1 in \cite{Villani2008}), and $\rho(y,t)$ is the solution of Fokker-Planck equation \eqref{fpe1} with $\sigma >0$. This implies that there exists a constant $T^*>0$ such that
			$$
			\int_{y \in \hat{U}} \rho(y, t) dy > \nu,
			$$
			for arbitrary $t>T^*$. It suggests that  there is a positive probability greater than $\nu$, that $Y(T_i)$ is in $\hat{U}$ when $(T_i-S_i)> T^*$ in the virtual diffusion process \eqref{SDE}. Because $\hat{Y} = Y(T_i)$ is used in \eqref{gradient2}, we conclude that the initial position $X(T_i) = \hat{Y}$ for the gradient flow \eqref{gradient} belongs to $\hat{U}$ with a positive probability.  The trajectories $X(t)$ that start from $X(T_i)$ are pushed into the neighborhood $U$ exponentially fast due to the definition of $\hat{U}$ and gradient flow properties. 
			
			In other words, the probability that the trajectory $X(t)$ does not end in $U$ is at most $(1-\nu)$ every time when one virtual diffusion and gradient flow cycle is completed. If such a cycle is performed $K$ times, the probability that $X(t)$ does not reach $U$ is $(1-\nu)^K$. Since $0 \leq (1-\nu)<1$, there exist a $K_0 > 0$ such that $(1-\nu)^K < \eta$ for any $K>K_0$. Therefore, 
			$$
			\mathbb{P}(X_{opt} \in U) \geq (1 - (1- \nu)^K) \geq (1-\eta),
			$$
			which completes the proof.
		\end{proof}
	}
	
	The proof also indicates that $(1-\eta)$ approaches $1$ in the manner of 
	$$\eta=\mathcal{O}((1-\nu)^K),$$
	which forms a geometric sequence in term\fix{s} of $K$. This is a much quick convergence rate than the logarithm function owned by the simulated annealing. \fix{We would like to point out that the convergence result presented in Theorem \ref{thm4} is in the sense of probability, which is different from the usual deterministic convergence results given in the $L^p$-norm or maximum norm,} but our numerical experiments show that $X_{opt}$ always reaches the desired shape $\Gamma$ without failure if the parameters are selected properly. 
	

	\section{Conclusions and Future Work}
	
	We present a motion planning strategy for a large group of robots to accomplish shape formation, one of the fundamental tasks in many applications that employ multi-robot systems. Typical challenges include how to avoid collisions and deadlocks in motion planning and how to achieve the desired shape with assurance. Those challenges become more significant for large groups of of robots and robots with low functionality. In our method, we calculate the individual robot trajectories by alternating two gradient flows that involve an attractive potential, a repelling function\fix{,} and a process of intermittent diffusion. The potential attracts robots to form the targeted shape, while the repelling function is designed to ensure collision-free motions. The intermittent diffusion, originally a stochastic approach but here realized by deterministic means, overcomes situations with deadlocks. Our strategy is inspired by recent developments in the theory of optimal transport which in turn provides the basis for theoretical guarantees of collision avoidance and global convergence. Numerical experiments confirm that the proposed algorithm is simple, yet effective in achieving \fix{the} desired objectives.  
	
	The presentation here in the two-dimensional setting can be extended to higher dimensions with straight forward adaptations. The proposed strategy can also be adapted to accommodate inhomogeneous multi-robot systems, in which robots may have different functionalities. In this scenario, the differences among robots must be reflected throughout the selections of the potential functions, including both $F(x)$ and $G(x)$. On the technical side, this may not be easy to accomplish and it is in our plan for further investigation.
	
	\newpage
	
	\bibliographystyle{plain}
	\bibliography{Mendeley}
	
\end{document}